\documentclass[a4paper,10pt]{amsart}

\usepackage[english]{babel}
\usepackage[utf8]{inputenc}

\usepackage{mathrsfs}
\usepackage{mathtools}
\usepackage{dsfont}
\usepackage{amssymb}
\usepackage{amsmath}
\usepackage{enumerate}
\usepackage{amsfonts}
\usepackage{amsthm,amsmath}
\usepackage{color}
\usepackage{hyperref}

\numberwithin{equation}{section}
\newcommand{\E}{\mathds{E}}
\newcommand{\N}{\mathds{N}}

\newcommand{\cL}{\mathscr{L}}

\newcommand{\cF}{\mathscr{F}}

\newcommand{\applied}[2]{\langle #1,#2\rangle}
\DeclarePairedDelimiter\norm{\lVert}{\rVert}
\DeclarePairedDelimiter\abs{\lvert}{\rvert}
\newcommand{\argument}{\,\cdot\,}
\renewcommand{\phi}{\varphi}

\newcommand{\colonequiv}{\mathrel{\mathop:}\Leftrightarrow}
\newcommand{\dx}{\;\mathrm{d}}
\newcommand{\eps}{\varepsilon}

\DeclareMathOperator{\supp}{supp}

\definecolor{mygreen}{rgb}{0.1,0.75,0.2}

\newcommand{\Sigmainv}{\Sigma_{\mathrm{inv}}}
\newcommand{\Sigmasim}{\Sigma/\hspace{-0.3em}\sim}

\theoremstyle{definition}
\newtheorem{definition}{Definition}[section]
\newtheorem{remark}[definition]{Remark}

\theoremstyle{plain}
\newtheorem{proposition}[definition]{Proposition}
\newtheorem{lemma}[definition]{Lemma}
\newtheorem{theorem}[definition]{Theorem}

\begin{document}

\title{Convergence of Dynamics and the Perron-Frobenius Operator}
\author{Moritz Gerlach}
\address{Moritz Gerlach\\Universit\"at Potsdam\\Institut f\"ur Mathematik\\Karl-Liebknecht-Stra{\ss}e 24--25\\14476 Potsdam\\Germany}
\email{moritz.gerlach@uni-potsdam.de}

\begin{abstract}
	We complete the picture how the asymptotic behavior of a dynamical system is reflected by properties of the associated Perron-Frobenius operator.
	Our main result states that strong convergence of the powers of the Perron-Frobenius operator is 
	equivalent to setwise convergence of the underlying dynamic in the measure algebra.
	This situation is furthermore characterized by uniform mixing-like properties of the system.
\end{abstract}

\thanks{The author thanks Jochen Gl\"uck for a fruitful discussion on the results of Section \ref{sec:mixing} 
during his stay at the University of Potsdam in August 2016.}

\date{\today}
\maketitle

\section{Introduction and Preliminaries}

Let $(\Omega,\Sigma,\mu;\phi)$ be a measure preserving dynamical system, i.e.\ a probability space $(\Omega,\Sigma,\mu)$ endowed with a measurable transformation
$\phi \colon \Omega \to \Omega$ such that the push-forward measure $\phi_*\mu$ equals $\mu$.
The so-called \emph{Perron-Frobenius operator} $P\in \cL(L^1(\Omega,\Sigma,\mu))$ can be defined by the Radon-Nikodym theorem as the unique linear
and positive operator satisfying
\begin{align}
\label{eqn:Pdef}
 \int_A Pf \dx\mu  = \int_{\phi^{-1}(A)} f \dx\mu 
\end{align}
for all $f\in L^1(\Omega,\Sigma,\mu)$ and $A\in \Sigma$. 
One easily checks that its adjoint $P^*$ is the associated \emph{Koopman operator} $T\in \cL(L^\infty(\Omega,\Sigma,\mu))$
given by $Tf \coloneqq f\circ \phi$, see \cite[Prop 4.3.1]{ding2009}. 
In particular, $P$ is a bi-Markov operator, i.e.\  $P$ is positive, $P\mathds{1}=\mathds{1}$ and $P^*\mathds{1}=\mathds{1}$.
Since the system is measure preserving, $T$ actually operates on every $L^p(\Omega,\Sigma,\mu)$-space and
$P$ can also be obtained as the adjoint of $T$ defined on $L^2(\Omega,\Sigma,\mu)$, extended to $L^1(\Omega,\Sigma,\mu)$.

It is well-known that some mixing properties of a dynamical system are described by the asymptotic behavior of the powers of the
associated Perron-Frobenius operator. We recall two examples from \cite[Thm 4.4.2]{ding2009}:
Firstly, the measure preserving dynamical system $(\Omega,\Sigma,\mu;\phi)$ is \emph{ergodic}, meaning that the measure of
each set $A\in \Sigma$ such that $\phi^{-1}(A)$ equals $A$ up to a nullset is either $0$ or $1$, if and only if the fixed space of $P$
consists of constant functions only. Since $P$ is a Dunford-Schwartz operator and therefore mean ergodic \cite[Thm 8.24]{haase2015},
ergodicity of the system is furthermore characterized by strong convergence of the Ces\`aro averages of $P$ to $\mathds{1}\otimes \mathds{1}$, the 
bi-Markov projection onto the constant functions.
Secondly, the system is \emph{(strongly) mixing}, i.e.\
\begin{align}
\label{eqn:mixing}
 \lim_{n\to\infty} \mu(\phi^{-n}(A)\cap B) = \mu(A)\mu(B) \text{ for all } A,B \in \Sigma,
 \end{align}
if and only if the sequence of powers $(P^n)_{n\in\N}$ converges in the weak operator topology to the just mentioned projection $\mathds{1}\otimes \mathds{1}$.

In this article we are concerned with \emph{strong convergence} of the sequence $(P^n)_{n\in\N}$, i.e.\ convergence with respect to the strong operator topology,
and its relation to the asymptotic behavior of the underlying measure preserving dynamical system.
It seems that this situation has never yet been characterized by a mixing-like property analogously to \eqref{eqn:mixing}.
Such a characterization is provided in Section \ref{sec:mixing}: while it is easy to see that strong convergence of $(P^n)_{n\in\N}$ to the 
rank-one projection $\mathds{1}\otimes \mathds{1}$ is equivalent to the fact that the convergence in \eqref{eqn:mixing} is uniform in $A\in \Sigma$,
we furthermore show that this is in turn equivalent to the formally weaker 
property that for any $B\in \Sigma$ there exists $D\in \Sigma$ with $\mu(D)>0$ such that
\begin{align}
\label{eqn:locallyuniformlymixing}
 \lim_{n\to\infty} \sup_{A\in \Sigma\text{, }A\subseteq D} \abs[\big]{\mu(\phi^{-n}(A)\cap B)  - \mu(A)\mu(B) } = 0.
 \end{align}
Moreover, this property  is also related to exactness of the system, which is defined as follows.
For each $n\in\N$ set $\Sigma_n \coloneqq \{ \phi^{-n}(A) : A \in \Sigma \}$. We  define the sub-$\sigma$-algebra 
$\Sigma_\infty \coloneqq \bigcap_{n\in\N} \Sigma_n$ and denote its completion within $\Sigma$ by $\overline{\Sigma_\infty}$.
The measure preserving system $(\Omega,\Sigma,\mu;\phi)$ is called \emph{exact} if the measure of each element of
$\Sigma_\infty$ is either $0$ or $1$.
It follows from \cite[Thm 3.2.3 and 4.4.2]{ding2009}
that in case where the measure preserving dynamical system is \emph{bimeasurable}, i.e.\
$\phi^{-1}(A),\phi(A) \in \Sigma$ for every $A\in \Sigma$, the following assertions are equivalent:
\begin{enumerate}[(i)]
	\item $\lim \mu(\phi^n(A))=1$ for each $A\in \Sigma$ with $\mu(A)>0$.
	\item The sequence $(P^n)_{n\in\N}$ converges strongly to $\mathds{1}\otimes \mathds{1}$.
	\item The system is exact.
\end{enumerate}
Note that the equivalence of (ii) and (iii) follows from \cite{lin1971}[Cor 4.1].
In view of this characterization, we call the Frobenius-Perron operator $P$ \emph{exact} if assertion (ii) holds.

As a consequence of these equivalences we obtain that any bimeasurable exact system is mixing and thus ergodic.
However, if the system $(\Omega,\Sigma,\mu;\phi)$ is not ergodic,
the powers $(P^n)_{n\in\N}$ can still converge to a projection onto the fixed space of $P$, which is no longer one-dimensional.
In Section \ref{sec:main} we extend the above equivalences  to non-ergodic systems by showing 
that in general strong convergence of $(P^n)_{n\in\N}$ is equivalent to setwise convergences of the iterates $(\phi^n)_{n\in\N}$ in the measure algebra.

In order to make the last statement precise, we have to introduce a little more notation. 
On the $\sigma$-algebra $\Sigma$ we define the equivalence relation
\[ A \sim B \colonequiv \mu(A\triangle B)=\mu(A\setminus B) + \mu(B\setminus A) =0 \]
and denote by $\Sigmasim$ the set of its equivalence classes. Note that $\Sigmasim$ is called the \emph{measure algebra},
which is a complete metric space with respect to the metric $d(A,B) \coloneqq \mu(A\triangle B)$, cf.\ \cite[Thm 1.12.6]{bogachev2007}.
For the sake of simplicity we omit the distinction in notation between a set $A\in \Sigma$ and its equivalence class in $\Sigmasim$.
Furthermore, we denote by 
\[ \Sigmainv\coloneqq \{ A\in \Sigma : A = \phi^{-1}(A) \}\]
the sub-$\sigma$-algebra of $\phi$-invariant sets and by $\overline{\Sigmainv}$ its completion within $\Sigma$.
Then clearly $\Sigmainv \subseteq \Sigma_\infty$ and therefore $\overline{\Sigmainv}\subseteq \overline{\Sigma_\infty}$.

For any $A\subseteq \Omega$ we define the \emph{minimal invariant superset} $A^*$ as the intersection of all $B\subseteq \Omega$ containing $A$ such
that $\phi^{-1}(B) = B$. If $\phi$ is bimeasurable, then the minimal invariant superset of any $A\in \Sigma$ is automatically measurable, 
see Lemma \ref{lem:A*measurable}.
Let us emphasize that our definition of $A^*$ is different from that in the literature, where
$A^*$ is usually defined as the minimal element of $\Sigmainv/\hspace{-0.3em}\sim$ that contains $A$ up to a nullset, cf.\ \cite[p.\ 21]{foguel1969}.
It is not difficult to see that both definitions coincide up to a nullset if and only if $\phi$ preserves nullsets.

Finally, let us denote by $\E(f| \Sigmainv)$ the conditional expectation of $f\in L^1(\Omega,\Sigma,\mu)$ with respect to $\Sigmainv$, i.e.\
the unique element in $L^1(\Omega,\Sigma,\mu)$ that contains a $\Sigmainv$-measurable representative and that 
\[ \int_B \E(f| \Sigmainv) \dx\mu = \int_B f \dx\mu \]
for all $B\in \Sigmainv$. It is also worth noting that $\E(f|\Sigmainv) = \E(f|\overline{\Sigmainv})$ in the sense of $L^1(\Omega,\Sigma,\mu)$.

Our main theorem now reads as follows.
\begin{theorem}
\label{thm:main}
Let $(\Omega,\Sigma,\mu;\phi)$ be a bimeasurable measure preserving dynamical system $(\Omega,\Sigma,\mu;\phi)$.
Then the following assertions are equivalent.
	\begin{enumerate}[(i)]
		\item For all $A\in \Sigma$ the sequence $(\phi^{n}(A))_{n\in\N}$ converges in $\Sigmasim$.
		\item For all $f\in L^1(\Omega,\Sigma,\mu)$ the sequence $(P^nf)_{n\in\N}$ converges in $L^1(\Omega,\Sigma,\mu)$.
		\item $\overline{\Sigma_\infty} = \overline{\Sigmainv}$.
	\end{enumerate}
	If these equivalent assertions hold, then the occurring limits can be identified as follows: 
	for each $f\in L^1(\Omega,\Sigma,\mu)$ the sequence $(P^nf)_{n\in\N}$ 
	converges to $\E(f| \Sigmainv)$ and for each $A\in \Sigma$ the sequence $(\phi^n(A))_{n\in\N}$  converges in $\Sigmasim$
	to the minimal invariant superset $A^*$.
\end{theorem}

Let us remark that the equivalence of (ii) and (iii) is not new, but is essentially contained in \cite{lin1971} and \cite{derriennic1976}.

It is also worth mentioning that in Theorem \ref{thm:main} $\phi$ is not assumed to preserve nullsets, 
i.e.\ modifying $A$ by a nullset will in general change the sequence $(\phi^n(A))$ and its limit in $\Sigmasim$.
Moreover, when starting with a transformation of the measure algebra $\phi^* \colon \Sigmasim \to \Sigmasim$ 
that induces an exact Perron-Frobenius operator and possesses bimeasurable liftings $\phi\colon \Omega \to \Omega$, 
Theorem \ref{thm:main} states that the convergence of $(\phi^n(A))$ does not depend on the choice of the 
lifting, whereas its limit does.
In order to illustrate this phenomenon consider the following simple example:
Let $\Omega\coloneqq\{1,2,3\}$ be endowed with the probability measure $\mu(1)\coloneqq \mu(3)\coloneqq 1/2$ and $\mu(2) \coloneqq 0$.
Then $\phi(1)\coloneqq 1$, $\phi(2)\coloneqq \phi(3) \coloneqq 3$ defines a bimeasurable measure preserving transformation $\phi\colon \Omega\to \Omega$.
For $A\coloneqq \{1,2\}$ we have $\phi^n(A) = \{1,3\}$ for every $n\in\N$ and therefore $\lim \phi^n(A) = \{1,3\}$ in $\Sigmasim$. 
Note that $\{1,3\}\sim\Omega$ and $\Omega$ is in fact the minimal invariant superset of $A$ with respect to the above definition.
On the other hand, $A\sim \{1\}$ and $\phi^n(\{1\}) = \{1\}$ for each $n\in\N$. However, $\{1\}$ is essentially different from $\Omega$.

Let us conclude this introduction with a short note on bimeasurability. 
It was proven in \cite{purves1966}, see also \cite{mauldin1981}, that a Borel-measurable function $\phi \colon X \to Y$
between two Polish spaces $X$ and $Y$ preserves Borel sets if and only if 
$\phi^{-1}(\{y\})$ is at most countable for all but at most countably many $y\in Y$.
For instance, the transformation $\phi(x) \coloneqq 2x \mod 1$ is
bimeasurable on $[0,1]$ with respect to the Borel-$\sigma$-algebra.
In some situations bimeasurability can be seen as a reasonable substitute for invertibility, cf.\ \cite{rice1978}.
However, if $\phi$ is bimeasurable and bijective, then $P=T^{-1}$ and the powers of $P$ can thus not be strongly convergent unless $P=I$.
In fact, if $P$ is strongly convergent, it follows from 
\[ \norm{f-Pf} \leq \norm{T^n}\cdot \norm{P^nf - P^{n+1}f} \to 0 \quad (n\to\infty) \]
that $Pf=f$ for any $f$.

\section{Uniform Mixing Properties}
\label{sec:mixing}

In the following we characterize strong convergence of the powers of the Perron-Frobenius operator $P$ by mixing-like properties
of the underlying measure preserving dynamical system $(\Omega,\Sigma,\mu;\phi)$.
In doing so, we make use of lower bound techniques originally invented by Lasota and Yorke in \cite{lasota1982, lasota1983}
to provide a sufficient condition for exactness of $P$. 
Ding extended this approach in \cite{ding2003} to characterize strong convergence of the sequence $(P^n)_{n\in\N}$ in
general. He proved that the limit $\lim_{n\to\infty} P^nf$ exists for any $f\in L^1(\Omega,\Sigma,\mu)$ if and only if for each $0 < f\in L^1(\Omega,\Sigma,\mu)$ 
there exists some 
$0 < h \in L^1(\Omega,\Sigma,\mu)$ such that 
\[ \lim_{n\to\infty} \norm{ (P^nf -h)^- }_{L^1} = 0,\]
where we use the notation $g=g^+ - g^-$ for the positive and negative part of a function $g\in L^1$.
This work has not yet gotten the attention it deserves and it acts as the key to the results in this section.
We also refer to \cite{gerlach2016} for a detailed discussion and further generalizations of lower bound techniques.

We start with a characterization of strong convergence of $(P^n)_{n\in\N}$ in general.
\begin{proposition}
\label{prop:convergencemixing}
Let $(\Omega,\Sigma,\mu;\phi)$ be a measure preserving dynamical system with associated Perron-Frobenius operator $P\in \cL(L^1(\Omega,\Sigma,\mu))$.
Then the following assertions are equivalent:
\begin{enumerate}[(i)]
\item The sequence $(P^n)_{n\in\N}$ converges with respect to the strong operator topology.
\item For each $B\in \Sigma$ with $\mu(B)>0$ there exists $D\in \Sigma$ with $\mu(D)>0$ and a constant $c>0$ such that
\[ \liminf_{n\to\infty} \inf_{A\in \Sigma} \bigr( \mu( \phi^{-n}(A) \cap B) - c\cdot \mu(D\cap A) \bigl) \geq 0.\]
\end{enumerate}
\end{proposition}
\begin{proof}
	(i) $\Rightarrow$ (ii): Let $(P^n)_{n\in\N}$ be strongly convergent and let $B\in \Sigma$ such that $\mu(B)>0$.
	Then the limit $f_B \coloneqq \lim_{n\to\infty} P^n \mathds{1}_B$
	is a positive function with $\norm{f_B}_{L^1} = \mu(B) >0$. Hence, there exists $D\in \Sigma$ with $\mu(D)>0$ and a constant $c>0$ such that
	$f_B \geq c\mathds{1}_D$; in particular $\lim \norm{(P^n \mathds{1}_B - c\mathds{1}_D)^-}_{L^1} = 0$.
	Since 
	\begin{align*}
	 &\inf_{A\in \Sigma} \bigr( \mu( \phi^{-n}(A) \cap B) - c\cdot \mu(D\cap A) \bigl) = \inf_{A\in \Sigma} \int_A \bigl( P^n \mathds{1}_B - c\mathds{1}_D\bigr) \dx\mu \\
	&\qquad \geq \inf_{A\in \Sigma} \int_A \bigl( P^n \mathds{1}_B - c\mathds{1}_D\bigr)^+ \dx\mu - \int_\Omega \bigl(P^n \mathds{1}_B - c\mathds{1}_D\bigr)^- \dx\mu
	\end{align*}
	for each $n\in\N$, assertion (ii) follows.

	(ii) $\Rightarrow$ (i): Let $0<f\in L^1(\Omega,\Sigma,\mu)$ and choose $B\in\Sigma$ with $\mu(B) >0$ and $d>0$ such that $f \geq d\mathds{1}_B$.
	Now pick $D \in \Sigma$ and $c>0$ according to assertion (ii). By \eqref{eqn:Pdef}, the definition of $P$, we obtain
	\begin{align*} \liminf_{n\to\infty} - \int_\Omega \bigl( P^n \mathds{1}_B - c\mathds{1}_D\bigr)^- 
	&= \liminf_{n\to\infty} \inf_{A\in \Sigma} \int_A \bigl( P^n \mathds{1}_B - c \mathds{1}_D \bigr) \dx\mu \\
	&= \liminf_{n\to\infty} \inf_{A\in \Sigma} \bigr( \mu( \phi^{-n}(A) \cap B) - c\cdot \mu(D\cap A) \bigl) \geq 0
	\end{align*}
	and hence that $\lim \norm{ (P^n \mathds{1}_B -c\mathds{1}_D)^-}_{L^1} = 0$. Now it follows from
	\[ \bigl( P^n f - dc\mathds{1}_D\bigr)^- \leq \bigl( P^n f - dP^n \mathds{1}_B\bigr)^- + \bigl( dP^n \mathds{1}_B - dc\mathds{1}_D\bigr)^- 
	= d \bigl(P^n \mathds{1}_B-c\mathds{1}_D\bigr)^-\]
	that $\lim \norm{(P^n f- cd\mathds{1}_D )^-}_{L^1}=0$. Therefore, we obtain from Ding's theorem \cite[Thm 1.1]{ding2003} that
	$(P^n)_{n\in\N}$ converges with respect to the strong operator topology.
\end{proof}

Next we characterize exactness of $P$ by a mixing-like properties of the dynamical system.

\begin{theorem}
\label{thm:exactmixing}
Let $(\Omega,\Sigma,\mu;\phi)$ be a measure preserving dynamical system with associated Perron-Frobenius operator $P\in \cL(L^1(\Omega,\Sigma,\mu))$.
Then the following assertions are equivalent:
\begin{enumerate}[(i)]
	\item The operator $P$ is exact, i.e.\ $\lim P^n f = \int_\Omega f \dx\mu \cdot \mathds{1}$ for each $f\in L^1(\Omega,\Sigma,\mu)$.
	\item For each $B\in \Sigma$ 
	\[ \lim_{n\to\infty}\sup_{A\in \Sigma} \abs[\big]{\mu(\phi^{-n}(A)\cap B) - \mu(A)\mu(B) } = 0.\]
	\item For each $B\in \Sigma$ there exists $D\in \Sigma$ with $\mu(D)>0$ such that 
	\[ \lim_{n\to\infty}\sup_{A\in \Sigma_D} \abs[\big]{\mu(\phi^{-n}(A)\cap B) - \mu(A)\mu(B) } = 0,\]
	where $\Sigma_D \coloneqq \{ A\cap D : A\in \Sigma \}$ denotes the trace $\sigma$-algebra.
\end{enumerate}
\end{theorem}
\begin{proof}
	(i) $\Rightarrow$ (ii): If $P$ is exact, then it follows from
	\begin{align*}
	\sup_{A\in \Sigma} \abs[\big]{\mu(\phi^{-n}(A)\cap B) - \mu(A)\mu(B) } &= \sup_{A\in \Sigma} \abs*{\int_A \bigr( P^n \mathds{1}_B - \mu(B) \mathds{1} \bigr) \dx\mu} \\
	&\leq \int_\Omega \abs[\big]{P^n \mathds{1}_B - \mu(B) \mathds{1}} \dx\mu
	\end{align*}
	that assertion (ii) holds.

	(ii) $\Rightarrow$ (iii): trivial.

	(iii) $\Rightarrow$ (i): Let $B \in \Sigma$ be arbitrary and let $D \in \Sigma$ according to assertion (iii).
	Then for a given $\eps>0$ we find $N\in\N$ such that
	\[ \abs[\big]{- \mu(\phi^{-n}(A)\cap B) + \mu(A)\mu(B)}  < \eps \text{ for all } n\geq N \text{ and } A\in \Sigma\text{, }A\subseteq D .\]
	In particular, we obtain that 
	\[  \mu(\phi^{-n}(A)\cap B) - \mu(A\cap D)\mu(B) > -\eps \text{ for all } n\geq N \text{ and } A\in \Sigma.\]
	Now it follows from Proposition \ref{prop:convergencemixing} that $(P^n)_{n\in\N}$ converges with respect to the strong operator topology.

	In order to prove that $P$ is exact, it remains to show that the fixed space of $P$ is one-dimensional and thus consists of constant functions only.
	In view of \cite[Thm 4.4.2]{ding2009} this is equivalent to ergodicity of the system.
	Let $\Omega_1 \in \Sigma$ be an invariant set, i.e.\ $\phi^{-1}(\Omega_1) \sim \Omega_1$,
	and note that $\Omega_2\coloneqq \Omega\setminus \Omega_1$ is also invariant by \cite[Lem 6.17]{haase2015}.
	Define $B\coloneqq \Omega_1$ and choose $D\in \Sigma$ with $\mu(D)>0$ according to assertion (iii).
	First, we consider the case that the set $A\coloneqq D\cap B$ has positive measure. Since $\phi^{-1}(A)$ is contained in $B$ up to a nullset, we obtain
	from assertion (iii) that
	\[ \mu(A) = \mu(\phi^{-n}(A)\cap B) \to \mu(A)\mu(B) \text{ as }n\to\infty.\]
	Therefore, $\mu(B)=1$ and thus $\mu(\Omega_2)=0$.
	Now assume that $D\cap B$ is a nullset. Since $A\coloneqq D$ is contained in $\Omega_2$ up to a nullset, we obtain from assertion (iii) that
	\[ \mu(\emptyset) = \mu(\phi^{-n}(A)\cap B) \to \mu(A) \mu(B) \text{ as }n\to\infty.\]
	Since $\mu(A)>0$ this implies that $\mu(B)=0$ and thus $\mu(\Omega_2)=1$.
	In both cases we have that $\mu(\Omega_1) \in \{0,1\}$ which proves that the system is ergodic. Therefore, $(P^n)_{n\in\N}$ converges in the strong
	operator topology to $\mathds{1}\otimes \mathds{1}$.
\end{proof}

Let us conclude this section with a remark on bimeasurable systems. 
If the system is mixing, it was shown in \cite[Thm 1]{rice1978} that
\begin{align}
\label{eqn:rice}
\lim_{n\to\infty} \mu(\phi^n(A)\cap B) = \lim_{n\to\infty}\mu(\phi^n(A))\mu(B)
\end{align}
for all $A,B\in \Sigma$.
We show next that the system is exact if and only if the convergence in \eqref{eqn:rice} holds
uniformly in $B\in \Sigma$. 

\begin{lemma}
\label{lem:riceexact}
Let $(\Omega,\Sigma,\mu;\phi)$ be a bimeasurable measure preserving dynamical system. Then $(\Omega,\Sigma,\mu;\phi)$ is exact
if and only if 
\begin{align}
\label{eqn:uniformrice}
\lim_{n\to\infty} \sup_{B\in \Sigma} \abs[\big]{ \mu(\phi^n(A)\cap B)  - \lim_{m\to\infty}\mu(\phi^m(A))\mu(B)} = 0 
\end{align}
for all $A\in \Sigma$.
\end{lemma}
\begin{proof}
First note that for every $A\in \Sigma$ it follows from 
$\phi^{-1}(\phi(A)) \supseteq A$ that $\mu(\phi(A)) = \mu(\phi^{-1}(\phi(A))) \geq \mu(A)$. Therefore,
the limit of the monotone sequence $\bigr(\mu(\phi^m(A))\bigl)_{m\in\N}$ always exists.

Now assume that the system is exact and let $A\in \Sigma$. If $\mu(\phi^n(A))=0$ for all $n\in\N_0$, 
then assertion \eqref{eqn:uniformrice} is trivial. Otherwise, we have $\lim \mu(\phi^n(A))=1$
by \cite[Thm 3.2.3]{ding2009} and therefore
\[ \sup_{B\in \Sigma} \abs[\big]{\mu(\phi^n(A)\cap B)- \mu(B)} = \sup_{B\in \Sigma} \mu(B\setminus \phi^n(A)) \leq \mu(\Omega\setminus \phi^n(A)) \to 0\]
as $n$ tends to infinity.

Conversely, assume that \eqref{eqn:uniformrice} holds. Fix $A\in \Sigma$, set $a\coloneqq \lim \mu(\phi^m(A))$ and
let $B_n \coloneqq\Omega\setminus \phi^n(A)$.
Then 
\[ \abs[\big]{\mu(\phi^n(A) \cap B_n) - a\cdot \mu(B_n) } = a\cdot \mu(\Omega\setminus \phi^n(A)) \to a(1-a)\]
as $n$ tends to infinity.
Hence, either $a=0$ and consequently $A$ is a nullset by the preliminary note, or $a=1$. 
By \cite[Thm 3.2.3]{ding2009}, this shows that the system is exact
\end{proof}

\section{Convergence of Dynamics}
\label{sec:main}

In this section we prove our main result, Theorem \ref{thm:main}, which is prepared by a series of lemmas
and preliminary remarks.

\begin{remark}
\label{rem:limitexpectation}
Let us recall that the Ces\`aro averages of the Perron-Frobenius operator $P$ corresponding to a 
measure preserving dynamical system $(\Omega,\Sigma,\mu;\phi)$ always converge strongly to the 
conditional expectation $\E(\argument | \overline{\Sigmainv})=\E(\argument |\Sigmainv)$, see e.g.\ \cite[Rem 13.24]{haase2015}.
This also implies that the Perron-Frobenius operator and the Koopman operator on $L^1(\Omega,\Sigma,\mu)$
share the same set of fixed points, namely $L^1(\Omega,\Sigmainv,\mu)$.
\end{remark}

\begin{lemma}
\label{lem:averagesexpectation}
	Let $(\Omega,\Sigma,\mu;\phi)$ be a measure preserving dynamical system with associated 
	Perron-Frobenius operator $P\in \cL(L^1(\Omega,\Sigma,\mu))$. 
	For each $A\in \Sigma$ and $m\in\N_0$ we have $\{ P^m \mathds{1}_A > 0 \} \subseteq \{ \E(\mathds{1}_A | \Sigmainv) > 0\}$ up to a nullset.
\end{lemma}
\begin{proof}
	Let $A\in \Sigma$ and fix a $\Sigmainv$-measurable representative of $\E(\mathds{1}_A | \Sigmainv)$.
	Then for $B\coloneqq \{ \E(\mathds{1}_A | \Sigmainv ) = 0 \} \in \Sigmainv$ we have
	\begin{align*}
	0 =\int_B \E(\mathds{1}_A | \Sigmainv) \dx\mu &\leftarrow \int_B \frac{1}{n} \sum_{m=0}^{n-1} P^m\mathds{1}_A \dx\mu 
	= \frac{1}{n} \sum_{m=0}^{n-1} \int_B P^m\mathds{1}_A \dx\mu \\
	&= \frac{1}{n} \sum_{m=0}^{n-1} \int_{\phi^{-m}(B)}\mathds{1}_A \dx\mu = \mu(A\cap B).
	\end{align*}
	Since the right-hand side does not depend on $n\in\N_0$, it follows that $\int_B P^m\mathds{1}_A \dx\mu = 0$ for all $m\in\N_0$ and therefore
	$\{ P^m \mathds{1}_A > 0\} \subseteq \{ \E(\mathds{1}_A | \Sigmainv) > 0\}$  up to a nullset.
\end{proof}

\begin{remark}
\label{rem:overlineSigmainv}
Let us briefly recall the well-known fact that for a measure preserving dynamical system $(\Omega,\Sigma,\mu;\phi)$ we have
\[ \overline{\Sigmainv} = \{ A \in \Sigma : A\sim \phi^{-1}(A) \}.\]
Indeed, every $A\in \overline{\Sigmainv}$ coincides with some set in $\Sigmainv$ up to a nullset and hence satisfies $\phi^{-1}(A)\sim A$.
Conversely, given a set $A\in\Sigma$ such that $A\sim \phi^{-1}(A)$, it can easily be checked that
\[ A \sim \bigcup_{n\in\N} \bigcap_{k\geq n} \phi^{-k}(A) \in \Sigmainv, \]
which implies that $A\in \overline{\Sigmainv}$.
\end{remark}

As mentioned in the introduction, we always have $\Sigmainv\subseteq \Sigma_\infty$, whereas the converse
inclusion does not hold in general. In fact, $\Sigma_\infty = \Sigma$ whenever $\phi$ is bimeasurable and bijective.
The next lemma gives a sufficient condition for $\overline{\Sigmainv} = \overline{\Sigma_\infty}$ to hold.

\begin{lemma}
\label{lem:sigmainftyinv}
	Let $(\Omega,\Sigma,\mu;\phi)$ be a bimeasurable measure preserving dynamical system such that
	the sequence $(\phi^n(A))$ converges in $\Sigmasim$. Then $\Sigma_\infty \subseteq \overline{\Sigmainv}$.
\end{lemma}
\begin{proof}
	Let $B\in \Sigma_\infty$. Then for every $n\in\N$ there exists $A_n \in\Sigma$ such that $B=\phi^{-n}(A_n)$.
	Consequently, $\phi^n(B) \subseteq A_n$ for all $n\in\N$. Since $B\subseteq \phi^{-n}(\phi^n(B))$ we have
	\[ \mu(\phi^n(B)) = \mu(\phi^{-n}(\phi^n(B))) \geq \mu(B) = \mu(\phi^{-n}(A_n))=\mu(A_n) \]
	for all $n\in\N$. This shows that $\phi^n(B) \sim A_n$ for each $n\in\N$. Hence, by assumption, $\lim A_n \eqqcolon A$ exists in $\Sigmasim$.
	Since $d(\phi^{-1}(C),\phi^{-1}(D)) = d(C,D)$ for all $C,D\in \Sigma$, the mapping
	$\phi^{-1} \colon \Sigmasim \to \Sigmasim$ is an isometry and in particular continuous.
	Hence,
	\begin{align*}
		d(\phi^{-1}(B),B) &= d\bigl(\phi^{-(n+1)}(A_n), \phi^{-n}(A_n)\bigr) \\
		&\leq d\bigl(\phi^{-(n+1)}(A_n),\phi^{-(n+1)}(A_{n+1})\bigr) + d\bigl(\phi^{-(n+1)}(A_{n+1}),\phi^{-n}(A_n)\bigr) \\
		&= d(A_n,A_{n+1}) + d(B,B) \to 0 \quad (n\to\infty).
	\end{align*}
	This shows that $B\sim \phi^{-1}(B)$ and thus, by Remark \ref{rem:overlineSigmainv}, $B\in \overline{\Sigmainv}$.
\end{proof}

The next lemma will be used to identify the limit of the sequence $(\phi^n(A))_{n\in\N}$ in $\Sigmasim$.
Beside this, it shows measurability of minimal invariant supersets.

\begin{lemma}
\label{lem:A*measurable}
	Let $(\Omega,\Sigma,\mu;\phi)$ be a bimeasurable measure preserving dynamical system.  Then the minimal invariant superset 
	of any $A\in \Sigma$  is itself measurable, i.e.\ $A^* \in \Sigmainv$, and is given by
	\[ A^* = \bigcup_{m\in\N_0} \phi^{-m} \biggl( \bigcup_{n\in\N_0} \phi^n(A) \biggr) \sim \bigcup_{n\in\N_0} \phi^n(A).\]
\end{lemma}
\begin{proof}
	Let $A\in \Sigma$ and set
	\[ B \coloneqq \bigcup_{n\in\N_0} \phi^n (A).\] 
	Then $A\subseteq B$ and obviously $\phi(B) \subseteq B$. Hence, the sequence $(\phi^{-n}(B))_{n\in\N_0}$
	is increasing and therefore,
	\[ D\coloneqq \bigcup_{n\in\N_0} \phi^{-n}(B)  \]
	is $\phi$-invariant, i.e.\ $\phi^{-1}(D) = D$.  Clearly, $A\subseteq D$ and $D\in \Sigmainv$.

	Now let $A\subseteq F\subseteq \Omega$ such that $\phi^{-1}(F)=F$. 
	Since $\phi^n(A) \subseteq \phi^n(F) = \phi^n(\phi^{-n}(F)) \subseteq F$ for every $n\in\N_0$, we know that $B\subseteq F$.
	Therefore, $\phi^{-n}(B) \subseteq \phi^{-n}(F) = F$ for every $n\in\N_0$, which implies that $D\subseteq F$. Since $F$ was arbitrary, this 
	show that $D$ is in fact the minimal invariant superset of $A$, denoted by $A^*$.

	Finally, we may conclude from the measure preservation that $\phi^{-1}(B) \sim B$ which implies that $D \sim B$.
\end{proof}

One implication of Theorem \ref{thm:main} is based on the following results by Lin from \cite[Thm 4.3]{lin1971};
see also \cite[Thm 2]{derriennic1976} for a generalization.

\begin{theorem}
\label{thm:Pnfconvergence}
	Let $X$ be a Banach space and $P\in \cL(X)$ be a contractive linear operator. We denote by $T\in \cL(X^*)$ the adjoint of $P$
	and set $B_{X^*}\coloneqq \{ g \in X^* : \norm{g}\leq 1\}$. Then for any $f\in X$ the following are equivalent:
	\begin{enumerate}[(i)]
		\item $\lim P^n f = 0$.
		\item $\applied{f}{h} = 0$ for all $h\in \cF\coloneqq \bigcap_{n\in\N} T^nB_{X^*}$.
	\end{enumerate}
\end{theorem}

Now we turn to the proof of Theorem \ref{thm:main}. 

\begin{proof}[Proof of Theorem \ref{thm:main}]
	(i) $\Rightarrow$ (iii): This implication follows from Lemma \ref{lem:sigmainftyinv} and the abovementioned remark.

	(iii) $\Rightarrow$ (ii): Assume that $\Sigma_\infty \subseteq \overline{\Sigmainv}$.
	Let $f\in L^1(\Omega,\Sigma,\mu)$ and define $\tilde f \coloneqq f-\E(f| \Sigmainv)$. By Remark \ref{rem:limitexpectation}
	we obtain that $\lim P^nf = \E(f| \Sigmainv)$ if and only if $\lim P^n \tilde f = 0$. 
	In view of Theorem \ref{thm:Pnfconvergence} it thus suffices to show that $\applied{\tilde f}{h} = 0$ for each $h\in \bigcap_{n\in\N}T^nB_{L^\infty}$.

	So let $h\in \bigcap T^n B_{L^\infty}$. Then for each $n\in\N$ there exists $g_n \in B_{L^\infty}$ such that $h = g_n \circ \phi^n$ in the sense of $L^\infty$.
	Let us fix $\Sigma$-measurable representatives of $h$ and all $g_n$ and let $A\in \Sigma$.
	Then for every $n\in\N$ we have $h^{-1}(A) \sim \phi^{-n}(g_n^{-1}(A)) \in \Sigma_n$. 
	Since $\Sigma_\infty \subseteq \overline{\Sigmainv}$ we obtain that
	\[ h^{-1}(A) \sim \bigcap_{n\in\N}  \phi^{-n}(g_n^{-1}(A)) \in \Sigma_\infty \subseteq \overline{\Sigmainv}.\]
	Due to the completeness of $\overline{\Sigmainv}$, this already implies that $h^{-1}(A) \in \overline{\Sigmainv}$.
	We hence showed that any $h\in \bigcap T^nB_{L^\infty}$ is $\overline{\Sigmainv}$-measurable.

	By definition of the conditional expectation we know that
	\[ \applied{\tilde f}{\mathds{1}_B} = \int_B f \dx\mu - \int_B \E(f| \Sigmainv) \dx\mu = 0 \]
	for all $B\in \overline{\Sigmainv}$. By linearity and density it thus follows that $\applied{\tilde f}{h} = 0$ for all $h\in \bigcap_{n\in\N} T^n B_{L^\infty}$.
	As said in the beginning, assertion (ii) now follows from Theorem \ref{thm:Pnfconvergence}.

	(ii) $\Rightarrow$ (i): Assume that $(P^n)_{n\in\N}$ is strongly convergent and let $A\in \Sigma$.
	We fix $\Sigma$-measurable representatives $f_n$ of $P^n\mathds{1}_A$ and a $\Sigmainv$-measurable representative $g$ of $\E(\mathds{1}_A | \Sigmainv)$.
	Let $S\coloneqq \supp g \coloneqq \{ g>0 \}$. For a given $\eps>0$ we may choose a constant $c>0$ small 
	enough such that $\mu(S\cap \{g< c\}) < \eps$. By assumption and Remark \ref{rem:limitexpectation} we 
	find $N\in\N$ such that $\norm{f_n -g}_{L^1} < \eps\cdot c$ for all $n\geq N$.
	Since $\{ f_n > 0 \} \subseteq \{g>0\}$ up to  a nullset by Lemma \ref{lem:averagesexpectation} and
	\[ c\cdot \mu(\{f_n = 0\}\cap \{ g\geq c\}) \leq \norm{f_n -g}_{L^1} < \eps\cdot c \]
	for all $n\geq N$, it follows that
	\begin{align*}
	&d(\{ f_n > 0\}, \{ g>0\} ) = \mu ( S \setminus \{f_n > 0\}) \\
	&\qquad = \mu(S\cap \{g< c\} \setminus \{f_n >0 \}) + \mu (S\cap \{g\geq c\} \setminus \{f_n>0\}) \leq 2\eps
	\end{align*}
	for all $n\geq N$. This shows that $\supp P^n\mathds{1}_A$ converges to $\supp \E(\mathds{1}_A | \Sigmainv)$ in $\Sigmasim$.
	Furthermore, since $P$ is a Markov operator and $A\subseteq \phi^{-n}(\phi^n(A))$, we have
	\[ \mu(A) = \int_\Omega P^n \mathds{1}_A \dx\mu \geq \int_{\phi^n(A)} P^n\mathds{1}_A \dx\mu  
	= \int_{\phi^{-n}(\phi^n(A))} \mathds{1}_A \dx\mu \geq \int_A\mathds{1}_A \dx\mu = \mu(A)\]
	for all $n\in\N_0$; this implies that $\supp P^n\mathds{1}_A \subseteq \phi^n(A)$ up to a nullset.

	In summary, we proved that $\phi^n(A)$ is (up to a nullset) a superset of the support of $P^n\mathds{1}_A$, 
	which converges to the support of $\E(\mathds{1}_A| \Sigmainv)$ in $\Sigmasim$.
	By Lemma \ref{lem:averagesexpectation}, we also know that $A$, which is nothing but the support of $P^0\mathds{1}_A$,
	is contained in the support of $\E(\mathds{1}_A | \Sigmainv)$ up to a nullset.
	Applying this observation to each of the sets $\phi^m(A)$, we conclude that for any $m\in \N_0$ and $\eps>0$ there exists
	$k\in \N$ such that $\mu\bigl(\phi^m(A) \setminus \phi^n(\phi^m(A))\bigr) < \eps$ whenever $n\geq k$.
	Now we show that the sequence $(\phi^n(A))_{n\in\N}$ converges to 
	\[ U \coloneqq \bigcup_{n\in\N_0} \phi^n(A)\]
	in $\Sigmasim$. Fix $\eps>0$ and $N\in\N$ such that
	\[ \mu\biggl( U \setminus \bigcup_{0\leq m\leq N} \phi^m(A) \biggr) < \eps.\]
	By what we have just seen, there exists  $k\in\N$ such that $\mu\bigl(\phi^m(A) \setminus \phi^{n+m}(A)\bigr)<\eps/N$ for
	every $n\geq k$ and all $1\leq m\leq N$. Therefore,
	\[ \mu\bigl( U \setminus \phi^n(A) \bigr)  \leq \mu\biggl(\bigcup_{1\leq m\leq N} \phi^m(A) \setminus \phi^n(A) \biggr) + \eps 
	\leq \eps + \sum_{m=1}^N \mu\bigl(\phi^m(A)\setminus \phi^n(A)\bigr) \leq 2\eps \]
	for all $n\geq k+N$. Since clearly $\phi^n(A) \subseteq U$ for every $n\in\N$, this proves that
	$\lim \phi^n(A) = U$ in $\Sigmasim$ as claimed.
	It now follows from Lemma \ref{lem:A*measurable} that 
	\[ U \sim \bigcup_{n\in\N_0} \phi^{-n}(U) = A^*, \]
	which completes the proof.
\end{proof}

\bibliographystyle{abbrv}
\bibliography{../analysis}

\end{document}